\documentclass{amsart}

\usepackage{amssymb}
\usepackage{amscd}

\newtheorem{theorem}{Theorem}[section]
\newtheorem{lemma}[theorem]{Lemma}
\newtheorem{proposition}[theorem]{Proposition}
\newtheorem{corollary}[theorem]{Corollary}

\theoremstyle{definition}
\newtheorem{definition}[theorem]{Definition}
\newtheorem{remark}[theorem]{Remark}
\newtheorem{notation}[theorem]{Notation}
\newtheorem{example}[theorem]{Example}

\numberwithin{equation}{section}

\newcommand{\co}{\mathop{\mathrm{co}}}
\newcommand{\cS}{\mathcal{S}}
\newcommand{\lhu}{\leftharpoonup}
\newcommand{\rhu}{\rightharpoonup}
\newcommand{\ga}{\mathfrak{a}}
\newcommand{\YD}{\mathcal{YD}}
\newcommand{\id}{\mathop{\mathrm{id}}}
\newcommand{\Ker}{\mathop{\mathrm{Ker}}}
\newcommand{\dotrtimes}{\mathop{\raisebox{0.2ex}{\makebox[0.9em][l]{${\scriptstyle >\joinrel\lessdot}$}}\raisebox{0.12ex}{$\shortmid$}}}
\newcommand{\cH}{\mathcal{H}}
\newcommand{\cA}{\mathcal{A}}
\newcommand{\gr}{\mathop{\mathrm{gr}}}
\newcommand{\cG}{\mathcal{G}}
\newcommand{\cM}{\mathcal{M}}
\newcommand{\cD}{\mathcal{D}}
\newcommand{\cDred}{\mathcal{D}_{\mathrm{red}}}
\newcommand{\diag}{\mathop{\mathrm{diag}}}
\newcommand{\bbz}{\mathbb{Z}}
\newcommand{\gh}{\mathfrak{h}}
\newcommand{\ord}{\mathop{\mathrm{ord}}}
\newcommand{\ch}{\mathop{\mathrm{ch}}}

\begin{document}

\title[quantized enveloping algebras]{Construction of quantized enveloping algebras by cocycle deformation}
\author{Akira Masuoka}  
\address{Institute of Mathematics, University of Tsukuba, Tsukuba, Ibaraki 305-8571 Japan}
\email{akira@math.tsukuba.ac.jp}


\begin{abstract}
By using cocycle deformation, we construct a certain class of Hopf algebras, containing the quantized enveloping algebras and their analogues, from what we call pre-Nichols algebras.
Our construction generalizes in some sense the known construction by (generalized) quantum doubles, but unlike in the known situation, it saves us from difficulties in checking complicated defining relations. 
\end{abstract}

\maketitle

\section*{Introduction}

When one studies the quantized enveloping algebras (\cite{Dr}, \cite{Ji}), one might have difficulties in checking complicated defining relations, or in encountering one after another, various analogues or variations, such as super-analogue, multiparameter version.
In this paper we wish to save ourselves from these difficulties, by using cocycle deformation \cite{D}.
Roughly speaking, our results will allow us to construct the quantized enveloping algebra $U_{q}$ on the tensor product $U_{q}^{-}\otimes U_{q}^{0}\otimes U_{q}^{+}$ of the familiar three subalgebras, only by giving product rules among distinct parts, such as
\begin{equation}
\label{eq0.1}
E_{i}F_{j}-F_{j}E_{i}=\delta_{ij}\frac{K_{i}-K_{i}^{-1}}{q_{i}-q_{i}^{-1}},
\end{equation}
where we have used the standard notation to denote standard generators in $U_{q}$.
We will see that by cocycle deformation, the relation (\ref{eq0.1}) is deformed from
\begin{equation}
\label{eq0.2}
E_{i}F_{j}-F_{j}E_{i}=0.
\end{equation}
It is known (see \cite[Section 3.2]{J}) that $U_{q}$ can be constructed as a certain quotient of the (generalized) quantum double of $U_{q}^{\leq 0}=U_{q}^{-}\otimes U_{q}^{0}$ and $U_{q}^{\geq 0}=U_{q}^{0}\otimes U_{q}^{+}$.
By the fact (see \cite{DT2}) that the quantum double construction is a special cocycle deformation, this last construction of $U_{q}$ is captured by our method, by which we are saved from checking the quantum Serre relations.
Indeed, our method does not concern so much individual defining relations, and can apply at the same time to various analogues and variations of $U_{q}$, and even to more general type of Hopf algebras.

Emphasizing a Hopf-Galois theoretic treatment, we discussed in \cite{M3} cocycle deformation of a certain class of pointed Hopf algebras, containing $U_{q}$, and proved that each of those Hopf algebras is a cocycle deformation of the naturally associated, simpler graded Hopf algebra.
This generalizes the preceding results by Didt \cite{Di}, and by Kassel and Schneider \cite{KS}.
In this paper we treat with a larger class of Hopf algebras including non-pointed ones, working in the context of pre-Nichols algebras; see below.

Throughout the paper we work over a fixed field $k$.
Section 1 is devoted to preliminaries from Hopf-Galois theory, including basics on cocycle deformation.
Recall here only that given a Hopf algebra $\cH$ and a $2$-cocycle $\sigma : \cH\otimes\cH\rightarrow k$, the {\em cocycle deformation} $\cH^{\sigma}$ of $\cH$ by $\sigma$ is the Hopf algebra which is constructed on the coalgebra $\cH$ with respect to the deformed product defined by (\ref{eq1.3}) below; see \cite{D}.

In Section 2, we define and discuss pre-Nichols algebras.
Given a Hopf algebra $H$ with bijective antipode, left Yetter-Drinfeld modules over $H$ form a braided tensor category ${}^{H}_{H}\YD$.
Given an object $V\in {}^{H}_{H}\YD$, we define a {\em pre-Nichols algebra} of $V$ to be a quotient $TV/I$ of the tensor algebra $TV$, which is naturally regarded as a graded braided Hopf algebra in ${}^{H}_{H}\YD$, by some homogeneous braided Hopf ideal $I$ such that $I\cap V=0$; see Definition \ref{2.1}.
The {\em Nichols algebra} of $V$, defined by Andruskiewitsch and Schneider [9,10], is precisely the pre-Nichols algebra with the largest possible $I$.
Since a pre-Nichols (especially, Nichols) algebra is an inductive limit of larger pre-Nichols algebras (see Proposition \ref{2.2}), it seems more convenient to treat with pre-Nichols algebras rather than to treat with Nichols algebras only, especially when we argue by induction; see the proof of Proposition \ref{3.9}.

Section 3 contains our main theorem.
For each $1\leq i\leq m$, let $V_{i}\in{}^{H}_{H}\YD$, and choose a pre-Nichols algebra $R_{i}$ of $V_{i}$.
Suppose that if $i\neq j$, the braidings
\[ c_{ij} : V_{i}\otimes V_{j}\xrightarrow{\simeq}V_{j}\otimes V_{i},\quad c_{ji} : V_{j}\otimes V_{i}\xrightarrow{\simeq} V_{i}\otimes V_{j} \]
are inverses of each other.
Then the braided tensor product $R:=R_{1}\otimes\dotsb\otimes R_{m}$ forms a graded braided Hopf algebra (in fact, a pre-Nichols algebra of $\bigoplus_{i=1}^{m}V_{i}$) in ${}^{H}_{H}\YD$, so that by bosonization \cite{R}, we have an ordinary graded Hopf algebra
\[ \cH=R\dotrtimes H. \]
Suppose that an $H$-linear map
\[ \lambda : \bigoplus_{i>j}[V_{i},V_{j}]\rightarrow k \]
is given, where $[V_{i},V_{j}]$ denotes the image $(\id-c_{ij})(V_{i}\otimes V_{j})$ of the braided commutator.
Associated to $\lambda$, a Hopf algebra $\cH^{\lambda}$ is defined; see Definition \ref{3.6}.
Our main theorem, Theorem \ref{3.10}, states especially that the Hopf algebra $\cH^{\lambda}$ is a cocycle deformation of $\cH$.
The relations in $\cH$
\begin{center}
the braided commutator $[v,w]=0$,
\end{center}
where $v\in V_{i}$, $w\in V_{j}$ with $i>j$, are deformed to the relations in $\cH^{\lambda}$
\[ [v,w]=\lambda[v,w]-v_{-1}w_{-1}\lambda[v_{0},w_{0}], \]
where $v\mapsto v_{-1}\otimes v_{0}$, $V_{i}\rightarrow H\otimes V_{i}$ denotes the $H$-comodule structure, just as (\ref{eq0.2}) is deformed to (\ref{eq0.1}) in $U_{q}$.
In Section 5, we reduce our results obtained in the preceding Sections 3,4 in the special situation when $H$ is the group Hopf algebra $k\Gamma$ of an abelian group $\Gamma$, and $V_{i}$ are of diagonal type (in the sense of [9,10]).
In particular, $U_{q}$ (and its analogues as well) are in that situation, and are presented as $\cH^{\lambda}$, when $H=U_{q}^{0}$, $m=2$, $R_{1}=$ the subalgebra of $U_{q}$ generated by $F_{i}K_{i}$, $R_{2}=U_{q}^{+}$ and $\lambda[E_{i},F_{j}K_{j}]=\delta_{ij}/(q_{i}-q_{i}^{-1})$.
The finite-dimensional pointed Hopf algebras $u(\cD,\lambda,\mu)$ defined by Andruskiewitsch and Schneider [9,10], which generalize Lustzig's small quantum groups \cite{L}, sit in the same situation, but $m$ should be a natural number in general.
In Corollary \ref{3.11} (to our main theorem), we prove by fully using a Hopf-Galois theoretic argument that the graded Hopf algebra $\gr\cH^{\lambda}$ arising from a natural filtration on $\cH^{\lambda}$ is isomorphic to $\cH$.

This paper is inspired by a couple of papers [13,14] by Radford and Schneider.
In Section 4 we apply our main theorem to prove \cite[Theorem 8.3]{RS1} on the quantum double of two bosonized Hopf algebras, in a slightly reformulated form; see Theorem 4.3.
The result is used to generalize \cite[Theorems 4.4, 4.11]{RS2} (see Theorem 5.3); these theorems of \cite{RS2} in turn generalize the quantum double construction of $U_{q}$, cited in the first paragraph above. \vspace{7pt}

\noindent
\textbf{Notation.} We work over a fixed field $k$ of arbitrary characteristic.
The unadorned $\otimes$ denotes the tensor product over $k$.
Suppose that $H$ is a Hopf algebra.
The coproduct and the counit are denoted by $\Delta : H\rightarrow H\otimes H$, $\varepsilon : H\rightarrow k$, respectively.
The antipode is denoted by $\cS : H\rightarrow H$ in script.
We choose the sigma notation such as
\[ \Delta(h)=h_{1}\otimes h_{2}\quad (h\in H) \]
for the coproduct, and such as
\[ \begin{array}{rl}
v\mapsto v_{-1}\otimes v_{0}, & V\rightarrow H\otimes V \vspace{5pt} \\
(\mbox{resp.,}\ \ v\mapsto v_{0}\otimes v_{1}, & V\rightarrow V\otimes H)
\end{array} \]
for the structure of a left (resp., right) $H$-comodule $V$.
For any left (resp., right) $H$-module $W$, the $H$-action is denoted by
\[ h\rhu w\quad (\mbox{resp.,}\ \ w\lhu h), \]
where $h\in H$, $w\in W$.
By gradings or filtrations, we always mean those indexed by the non-negative integers $\{0,1,2,\dotsc\}$.

\section{Preliminaries on Galois and cleft objects}

Let $H$ be a Hopf algebra.
A right $H$-comodule algebra $A\neq 0$ is called a {\em right $H$-Galois object} if the map
\[ A\otimes A\rightarrow A\otimes H,\quad a\otimes b\mapsto ab_{0}\otimes b_{1} \]
is bijective.
In this case the subalgebra of $H$-coinvariants
\[ A^{\co H}:=\{a\in A\;|\;a_{0}\otimes a_{1}=a\otimes 1\} \]
in $A$ necessarily coincides with $k$.
Let $L$ be another Hopf algebra.
A {\em left $L$-Galois object} is defined analogously for a left $L$-comodule algebra.
An {\em $(L,H)$-biGalois object} is an $(L,H)$-bicomodule algebra which is Galois on both sides.
As was proved by Schauenburg \cite{Sb}, given a right $H$-Galois object $A$,
there exist uniquely (up to isomorphism) a Hopf algebra $L$ together with a coaction $A\rightarrow L\otimes A$ which makes $A$ into an $(L,H)$-biGalois object (though we will not use this result).

To give a natural way of constructing biGalois object, let $G$ be a Hopf algebra, and let $\gamma : G\rightarrow k$ be an algebra map.
The inverse of $\gamma$ in the group (under the convolution product) of all algebra maps $G\rightarrow k$ is given by $\gamma^{-1}=\gamma\circ\cS$.
We have an algebra automorphism of $G$,
\[ x\mapsto\ x\lhu\gamma:=\gamma(x_{1})x_{2}, \]
and a Hopf algebra automorphism of $G$,
\[ x\mapsto\ \gamma^{-1}\rhu x\lhu\gamma:=\gamma(x_{1})x_{2}\gamma^{-1}(x_{3}). \]
Let $\ga\subset G$ be a Hopf ideal.
Then we have an ideal $\ga\lhu\gamma$, and a Hopf ideal $\gamma^{-1}\rhu\ga\lhu\gamma$ in $G$.
Let $F$ be a Hopf algebra given a Hopf algebra map $\iota : G\rightarrow F$.
Let
\[ I_{L}=F(\gamma^{-1}\rhu\ga\lhu\gamma)F,\quad I=F(\ga\lhu\gamma)F,\quad I_{R}=F\ga F \]
denote the ideals of $F$ which are generated by the images of $\gamma^{-1}\rhu\ga\lhu\gamma$, $\ga\lhu\gamma$, $\ga$, respectively.
Set
\begin{equation}
\label{eq1.1}
L=F/I_{L},\quad A=F/I,\quad H=F/I_{R}.
\end{equation}

\begin{proposition}
\label{1.1}
(1) $I_{L}$ and $I_{R}$ are Hopf ideals of $F$, whence $L$ and $H$ are Hopf algebras.

(2) The coproduct of $F$ induces algebra maps
\[ L\otimes A\leftarrow A\rightarrow A\otimes H, \]
by which $A$ is an $(L,H)$-bicomodule algebra.

(3) If $I\neq F$ (or $A\neq 0$), $A$ is an $(L,H)$-biGalois object.
\end{proposition}

This was stated in \cite[Theorem 2]{M2} in the special situation that $\iota$ is an inclusion, but the assumption given in (3) above was missing.
The corrected statement was given in \cite[Theorem 3.4]{BDR}.

\begin{proof}
It is easy to see (1), (2).
For (3), we see as in the proof of \cite[Theorem 2]{M2} that the isomorphism
\[ F\otimes F\xrightarrow{\simeq}F\otimes F,\quad a\otimes b\mapsto ab_{1}\otimes b_{2}, \]
whose inverse is given by $a\otimes b\mapsto a\cS(b_{1})\otimes b_{2}$, induces an isomorphism $A\otimes A\xrightarrow{\simeq}A\otimes H$.
This proves that if $A\neq 0$, $A$ is right $H$-Galois.
To see that $A$ is then left $L$-Galois, use the isomorphism
\[ F\otimes F\xrightarrow{\simeq}F\otimes F,\quad a\otimes b\mapsto a_{1}\otimes a_{2}b, \]
whose inverse is given by $a\otimes b\mapsto a_{1}\otimes \cS(a_{2})b$.
\end{proof}

Let $H$ be a Hopf algebra.
A right {\em $H$-cleft object} is a right $H$-comodule algebra $A$ which admits such an $H$-colinear isomorphism $\phi : H\xrightarrow{\simeq}A$ that is invertible with respect to the convolution product.
Such a $\phi$ can be chosen so as to preserve unit;
in this case, $\phi$ is called a {\em section} \cite{DT1}.
A right $H$-cleft object is characterized as such a right $H$-Galois object that admits an $H$-colinear isomorphism $H\xrightarrow{\simeq}A$ \cite[Theorem 9]{DT1}.
Given another Hopf algebra $L$, a {\em left $L$-cleft object} and an {\em $(L,H)$-bicleft object} are defined analogously.
An invertible linear map $\sigma : H\otimes H\rightarrow k$ is called a {\em $2$-cocycle}, if
\[ \begin{array}{c}
\sigma(g_{1},h_{1})\sigma(g_{2}h_{2},l)=\sigma(h_{1},l_{1})\sigma(g,h_{2}l_{2}), \vspace{5pt} \\
\sigma(h,1)=\varepsilon(h)=\sigma(1,h)
\end{array} \]
for all $g,h,l\in H$.
Given a pair $(A,\phi)$ of a right $H$-cleft object $A$ together with a section $\phi : H\xrightarrow{\simeq}A$, the linear map $\sigma$ defined on $H\otimes H$ by
\begin{equation}
\label{eq1.2}
\sigma(g,h):=\phi(g_{1})\phi(h_{1})\phi^{-1}(g_{2}h_{2})\quad (g,h\in H)
\end{equation}
takes values in $k=A^{\co H}$, and $\sigma : H\otimes H\rightarrow k$ turns into a $2$-cocycle \cite[Theorem 11]{DT1}. \\

\noindent
{\rm\bf Proposition 1.2 [19, Proposition 1.4].}
{\it Every $2$-cocycle $\sigma$ arises in this way, uniquely (up to isomorphism) from a pair $(A,\phi)$.}
\addtocounter{theorem}{1} \\

In the situation above, let ${}_{\sigma}H$ denote the right $H$-comodule $H$ endowed with the original unit and the deformed product
\[ g\cdot h:=\sigma(g_{1},h_{1})g_{2}h_{2}\quad (g,h\in H). \]
Then, ${}_{\sigma}H$ is a right $H$-cleft object such that ${}_{\sigma}H\rightarrow A$, $h\mapsto\phi(h)$ is an $H$-colinear algebra isomorphism.
This ${}_{\sigma}H$ is called the {\em $H$-crossed product} over $k$ implemented by $\sigma$.

Given a $2$-cocycle $\sigma : H\otimes H\rightarrow k$, let $H^{\sigma}$ denote the coalgebra $H$ endowed with the original unit and the deformed product
\begin{equation}
\label{eq1.3}
g\cdot h:=\sigma(g_{1},h_{1})g_{2}h_{2}\sigma^{-1}(g_{3},h_{3})\quad (g,h\in H).
\end{equation}

\noindent
{\rm\bf Proposition 1.3 [3, Theorem 1.6].}
{\it $H^{\sigma}$ is a bialgebra, and is indeed a Hopf algebra with respect to the deformed antipode $\cS^{\sigma}$ as given in \cite[p.\ 1736, line $-4$]{D}.}
\addtocounter{theorem}{1}

\begin{definition}
\label{1.4}
We call $H^{\sigma}$ the {\em cocycle deformation} of $H$ by $\sigma$.
\end{definition}

We remark that $\sigma^{-1}$ is a $2$-cocycle for $H^{\sigma}$, such that $(H^{\sigma})^{\sigma^{-1}}=H$.
So, we can say that $H$ and $H^{\sigma}$ are cocycle deformations of each other.

Keep $\sigma$ denoting a $2$-cocycle.
The $(H,H)$-bicomodule structure on $H$ makes ${}_{\sigma}H$ into an $(H^{\sigma},H)$-bicleft object.

\begin{lemma}
\label{1.5}
Suppose that $\sigma$ arises from a pair $(A,\phi)$, and suppose that this last $A$ is also a left $L$-Galois object, where $L$ is another Hopf algebra.
Suppose that $\eta : H\rightarrow L$ is a coalgebra map which makes the diagram
\[ \begin{CD}
H @>{\Delta}>> H\otimes H \\
@V{\phi}VV @VV{\eta\otimes\phi}V \\
A @>>> L\otimes A
\end{CD} \]
commute, where the arrow on the bottom denotes the $L$-comodule structure on $A$.
Then, $\eta$ is necessarily bijective, and further turns into a Hopf algebra isomorphism $H^{\sigma}\xrightarrow{\simeq}L$, so that $L$ is the cocycle deformation of $H$ by $\sigma$.
\end{lemma}
\begin{proof}
The commutative diagram above gives rise to
\begin{center}
\begin{picture}(130,70)
\put(10,55){${}_{\sigma}H\otimes {}_{\sigma}H$}
\put(55,57){\vector(1,0){30}}
\put(65,60){$\simeq$}
\put(90,55){$H^{\sigma}\otimes {}_{\sigma}H$}
\put(30,50){\vector(0,-1){30}}
\put(3,35){$\phi\otimes\phi$}
\put(35,35){$\simeq$}
\put(17,8){$A\otimes A$}
\put(55,10){\vector(1,0){30}}
\put(65,13){$\simeq$}
\put(97,8){$L\otimes A$,}
\put(110,50){\vector(0,-1){30}}
\put(115,35){$\eta\otimes\phi$}
\put(65,30){\large $\circlearrowright$}
\end{picture}
\end{center}
in which the horizontal isomorphisms are given by $g\otimes h\mapsto g_{1}\otimes g_{2}\cdot h$, $a\otimes b\mapsto a_{-1}\otimes a_{0}b$, respectively.
It results that $\eta$ is bijective.
We see as in the proof of \cite[Lemma 1.5]{M3} that $\eta : H^{\sigma}\xrightarrow{\simeq} L$ is a Hopf algebra isomorphism, since in order for $\Delta : {}_{\sigma}H\rightarrow H\otimes {}_{\sigma}H$ to be an algebra map, the algebra structure on the coalgebra $H$ must coincide with that structure on $H^{\sigma}$.
\end{proof}

Let $J,K$ be Hopf algebras.
We say in general that a Hopf algebra $G$ including $J,K$ as Hopf subalgebras {\em factorizes into} $J,K$ if the product map $J\otimes K\rightarrow G$ is an isomorphism.
In this case, $G$ is the tensor product $J\otimes K$ as a coalgebra, and the algebra structure is determined if we present the products $x\cdot a$, where $x\in K$, $a\in J$, as elements in $J\otimes K$;
we will call such presentations {\em product rules}.
It is enough to take $x,a$ from respective systems of generators.

A linear map $\tau : J\otimes K\rightarrow k$ is called a {\em skew pairing} \cite[Definition 1.3]{DT2}, if
\[ \begin{array}{l}
\tau(ab,x)=\tau(a,x_{1})\tau(b,x_{2}), \vspace{5pt} \\
\tau(a,xy)=\tau(a_{1},y)\tau(a_{2},x), \vspace{5pt} \\
\tau(1,x)=\varepsilon(x),\quad \tau(a,1)=\varepsilon(a)
\end{array} \]
for all $a,b\in J$, $x,y\in K$.
Such a $\tau$ is necessarily invertible, and the inverse is given by
\[ \tau^{-1}(a,x)=\tau(\cS(a),x)\quad (a\in J,\ x\in K), \]
which further equals $\tau(a,\cS^{-1}(x))$ if the antipode of $K$ is bijective \cite[Lemma 1.4]{DT2}.
Define $H=J\otimes K$, the tensor-product Hopf algebra.
If $\tau : J\otimes K\rightarrow k$ is a skew pairing, then the linear map $\sigma : H\otimes H\rightarrow k$ defined by 
\begin{equation}
\label{eq1.4}
\sigma(a\otimes x,b\otimes y)=\varepsilon(a)\tau(b,x)\varepsilon(y)\quad (a,b\in J,\ x,y\in K)
\end{equation}
is a $2$-cocycle \cite[Proposition 1.5]{DT2}.
We see easily the following.

\begin{proposition}
\label{1.6}
Let $H=J\otimes K$ be as above.
Let $\sigma : H\otimes H\rightarrow k$ be a $2$-cocycle, and suppose that it arises from a pair $(A,\phi)$.
Then, $\sigma$ arises as above from some skew pairing $\tau : J\otimes K\rightarrow k$ if and only if $\phi$ is an algebra map restricted to $J=J\otimes k$, and to $K=k\otimes K$.
In this case the cocycle deformation $H^{\sigma}$ ($=J\otimes K$) includes $J$ ($=J\otimes k$) and $K$ ($=k\otimes K$) as Hopf subalgebras. Moreover, $H^{\sigma}$ factorizes into $J,K$, and obeys the product rule
\[ x\cdot a=\tau(a_{1},x_{1})a_{2}\otimes x_{2}\tau(\cS(a_{3}),x_{3})\quad (a\in J,\ x\in K). \]
\end{proposition}

\begin{remark}
\label{1.7}
As was pointed out by Doi and Takeuchi \cite[Remark 2.3]{DT2}, the thus constructed $H^{\sigma}$ coincides with the quantum double formulated by DeConcini and Lyubashenko \cite{DeL} (see also \cite[Section 3.2]{J}) which generalizes Drinfeld's double \cite{Dr}.
\end{remark}

\section{Pre-Nichols algebras}

Let $H$ be a Hopf algebra with bijective antipode.
Let ${}^{H}_{H}\YD$ denote the braided tensor category of left Yetter-Drinfeld modules over $H$ \cite[p.\ 213]{Mo}.
Let $V\in {}^{H}_{H}\YD$.
Thus, $V$ is a left $H$-module given a left $H$-comodule structure $\rho : V\rightarrow H\otimes V$, $\rho(v)=v_{-1}\otimes v_{0}$ such that
\[ \rho(h\rhu v)=h_{1}v_{-1}\cS(h_{3})\otimes(h_{2}\rhu v_{0})\quad (h\in H,\ v\in V). \]
Given another object $W\in {}^{H}_{H}\YD$, the braiding is given by
\[ c=c_{V,W} : V\otimes W\xrightarrow{\simeq}W\otimes V,\quad c(v\otimes w)=(v_{-1}\rhu w)\otimes v_{0}. \]
We say that $V$ and $W$ are {\em symmetric} if $c_{W,V}\circ c_{V,W}=\id_{V\otimes W}$, the identity map, or equivalently if $c_{V,W}\circ c_{W,V}=\id_{W\otimes V}$.

\begin{definition}
\label{2.1}
A {\em pre-Nichols algebra} of $V$ is a graded braided bialgebra $R=\bigoplus_{n=0}^{\infty} R(n)$ in ${}^{H}_{H}\YD$ such that
\begin{enumerate}
\renewcommand{\labelenumi}{(\roman{enumi})}
\item $R(0)=k$, the unit object in ${}^{H}_{H}\YD$,
\item $R(1)=V$ as objects in ${}^{H}_{H}\YD$, and
\item $R$ is generated by $R(1)$.
\end{enumerate}
\end{definition}

Let $R$ be such as defined above.
By (i), the coradical of $R$ is $k$ ($=R(0)$), so that $R$ has a (necessarily graded) antipode, or in other words, $R$ is a graded braided Hopf algebra.
One also sees that $V$ ($=B(1)$) consists of primitives, or namely
\[ V\subset P(R):=\{v\in R\;|\;\Delta(v)=v\otimes 1+1\otimes v\}. \]
$R$ is called the {\em Nichols algebra} [9,10] of $V$ if in addition,
\begin{enumerate}
\renewcommand{\labelenumi}{(\roman{enumi})}
\addtocounter{enumi}{3}
\item $V=P(R)$.
\end{enumerate}

The tensor algebra $TV$ of $V$ is naturally a pre-Nichols algebra of $V$.
The algebra map $\varphi : TV\rightarrow R$ induced from the inclusion $V\hookrightarrow R$ is a graded braided Hopf algebra epimorphism.
We call the kernel $\Ker\varphi$ is the {\em defining ideal} of $R$, or we say that $R$ is {\em defined by} the ideal.
The ideal is a homogeneous braided bi-ideal (necessarily, Hopf ideal) which trivially intersects with $V$.
Conversely, if $I\subset TV$ is such a bi-ideal, then $TV/I$ is a pre-Nichols algebra defined by $I$.
The Nichols algebra of $V$ is precisely the pre-Nichols algebra of $V$ with the largest defining ideal; see \cite[Proposition 2.2]{AS1}.

Let $R$ be a pre-Nichols algebra of $V$ with defining ideal $I$.
Set $I_{0}=0$, and define inductively an ascending chain
\begin{equation}
\label{eq2.1}
0=I_{0}\subset I_{1}\subset I_{2}\subset\dotsb\ (\subset I)
\end{equation}
of homogeneous braided bi-ideals of $TV$ included in $I$, so that $I_{n+1}/I_{n}$ is the ideal in $TV/I_{n}$ generated by all primitives included in $I/I_{n}$.
Notice that those primitives form a graded subobject of $I/I_{n}$ in ${}^{H}_{H}\YD$.
Then one sees that (\ref{eq2.1}) consists of homogeneous braided bi-ideals.

\begin{proposition}
\label{2.2}
We have $I=\bigcup_{n=0}^{\infty}I_{n}$.
It follows that
\[ R=TV/I=\varinjlim TV/I_{n}, \]
the inductive limit along the projections $TV/I_{n}\rightarrow TV/I_{n+1}$.
\end{proposition}
\begin{proof}
Set $J=\bigcup_{n=0}^{\infty}I_{n}$.
Then, $J\subset I$.
We wish to show that the canonical $TV/J\rightarrow R$ is injective.
It suffices to prove that the map is injective, restricted to the primitives.
Suppose that an element $u\in TV$ which is a primitive modulo $J$ is contained in $I$.
Then, $u$ is a primitive modulo some $I_{n}$, which implies $u\in I_{n+1}\subset J$.
This proves the desired injectivity.
\end{proof}

\begin{definition}
\label{2.3}
Given $V,W\in{}^{H}_{H}\YD$, let $[V,W]$ denote the image of the braided commutator
\[ \id-c_{V,W} : V\otimes W\rightarrow (V\otimes W)\oplus (W\otimes V). \]
Since this last is a morphism in ${}^{H}_{H}\YD$, $[V,W]$ is a subobject of $T^{2}(V\oplus W)$ in ${}^{H}_{H}\YD$.
For $v\in V$, $w\in W$, we write $[v,w]$ for $(\id-c)(v\otimes w)$, or explicitly
\[ [v,w]=v\otimes w-(v_{-1}\rhu w)\otimes v_{0}. \]
\end{definition}

Let $V,W\in{}^{H}_{H}\YD$.
Let $R,S$ be pre-Nichols algebras of $V,W$ with defining ideals $I,J$, respectively.
Then the braided tensor product $R\otimes S$ is naturally a graded algebra and coalgebra in ${}^{H}_{H}\YD$.

\begin{proposition}
\label{2.4}
The following are equivalent:
\begin{enumerate}
\renewcommand{\labelenumi}{(\roman{enumi})}
\item $V$ and $W$ are symmetric (see just above Definition \ref{2.1});
\item $[V,W]=[W,V]$;
\item $[V,W]$ consists of primitives in $T(V\oplus W)$;
\item $R\otimes S$ is a braided bialgebra.
\end{enumerate}
If these conditions hold, $R\otimes S$ is a pre-Nichols algebra of $V\oplus W$ defined by the ideal generated by $I,J$ and $[V,W]$.
\end{proposition}
\begin{proof}
It is easy to (i) $\Leftrightarrow$ (ii). We only remark that under (i),
\[ [v,w]=-[v_{-1}\rhu w,v_{0}]\quad (v\in V,\ w\in W). \]
As is well-known, (i) $\Leftrightarrow$ (iii) follows from
\begin{equation}
\label{eq2.2}
\Delta[v,w]=[v,w]\otimes 1+1\otimes[v,w]+(\id-c^{2})(v\otimes w).
\end{equation}
If (i) holds, the braidings $R\otimes S\rightleftarrows S\otimes R$ are inverses of each other, which implies (iv) as a general fact.
If (iv) holds, then the algebra map $\varphi : T(V\oplus W)\rightarrow R\otimes S$ induced from the inclusion $V\oplus W=(R\otimes S)(1)\hookrightarrow R\otimes S$ is a coalgebra map.
Since $\varphi[w,v]=0$ by the product rule on $R\otimes S$, one sees from (\ref{eq2.2}) that $(\id-c^{2})(w\otimes v)=0$ in $S(1)\otimes R(1)$, whence (i) holds.

Suppose that the conditions hold.
Let $P\subset T(V\oplus W)$ denote the ideal generated by $I,J$ and $[V,W]$.
By (iii), this is a homogeneous braided bi-ideal.
Set $B=T(V\oplus W)/P$.
We see that the map $\varphi$ above induces a braided Hopf algebra map $\bar{\varphi} : B\rightarrow R\otimes S$.
The tensor product of the natural braided Hopf algebra maps $R\rightarrow B$, $S\rightarrow B$, composed with the product $B\otimes B\rightarrow B$, defines $\nu : R\otimes S\rightarrow B$.
Since $\bar{\varphi}\circ\nu=\id$, $\nu$ is injective.
It is also surjective, since $[V,W]\subset P$.
Thus, $\nu$ and hence $\bar{\varphi}$ are isomorphisms, which proves the last statement.
\end{proof}

\begin{remark}
\label{2.5}
Let $B$ be an arbitrary braided bialgebra in ${}^{H}_{H}\YD$.
As is seen from (\ref{eq2.2}), if $a,b$ are primitives in $B$ such that $c^{2}(a\otimes b)=a\otimes b$, then $[a,b]$ is a primitive.
Here we have used the slightly abused notation
\begin{equation}
\label{eq2.3}
[a,b]=ab-(a_{-1}\rhu b)a_{0}.
\end{equation}
\end{remark}

\section{Constructing Hopf algebras from pre-Nichols algebras}

\begin{notation}
\label{3.1}
Throughout in what follows, $H$ denotes a Hopf algebra with bijective antipode.
Let $\{V_{\alpha}\}_{\alpha\in\Lambda}$ be a family of objects in ${}^{H}_{H}\YD$ such that for all $\alpha\neq\beta$ in $\Lambda$, $V_{\alpha}$ and $V_{\beta}$ are symmetric.
Set
\[ V=\bigoplus_{\alpha\in\Lambda}V_{\alpha}\ \ \mbox{in}\ \ {}^{H}_{H}\YD. \]
Introduce a total order onto $\Lambda$, and set
\[ Z=\bigoplus_{\alpha>\beta}[V_{\alpha},V_{\beta}]\ \ (\mbox{in}\ \ T^{2}V). \]
See Definition \ref{2.3}.
By Proposition \ref{2.4} (see (ii)), this is independent of choice of the total order on $\Lambda$.
Choose an $H$-linear map
\[ \lambda : Z\rightarrow k. \]
Thus, $\lambda(h\rhu z)=\varepsilon(h)\lambda(z)$ for $h\in H$, $z\in Z$.
For simplicity we will write $\lambda[v,w]$ for $\lambda([v,w])$.
\end{notation}

\begin{remark}
\label{3.2}
Through the identification $V_{\alpha}\otimes V_{\beta}=[V_{\alpha},V_{\beta}]$ (in ${}^{H}_{H}\YD$) given by the isomorphism $\id-c$, $\lambda$ may be regarded as a family of $H$-linear maps $V_{\alpha}\otimes V_{\beta}\rightarrow k$ ($\alpha>\beta$).
\end{remark}

From the braided Hopf algebra $TZ$ in ${}^{H}_{H}\YD$, construct the ordinary Hopf algebra $TZ\dotrtimes H$ by bosonization \cite{R} (or biproduct construction), and let
\begin{equation}
\label{eq3.1}
\ga=(Z)
\end{equation}
denote the ideal of $TZ\dotrtimes H$ generated by $Z=T^{1}Z$;
this is in fact a Hopf ideal.
Since $\lambda$ is $H$-linear, the induced algebra map $TZ\rightarrow k$, which we denote by $\tilde{\lambda}$, is $H$-linear.
Therefore, $\tilde{\lambda}$ extends uniquely to the algebra map
\begin{equation}
\label{eq3.2}
\gamma : TZ\dotrtimes H\rightarrow k,\quad \gamma(a\otimes h)=\tilde{\lambda}(a)\varepsilon(h).
\end{equation}
Notice from Section 1 that we have an ideal $\ga\lhu\gamma$, and a Hopf ideal $\gamma^{-1}\rhu\ga\lhu\gamma$ in $TZ\dotrtimes H$.

\begin{lemma}
\label{3.3}
(1) $\ga\lhu\gamma$ is generated by the elements
\begin{equation}
\label{eq3.3}
[v,w]+\lambda[v,w],
\end{equation}
where $v\in V_{\alpha}$, $w\in V_{\beta}$ with $\alpha>\beta$.

(2) $\gamma^{-1}\rhu\ga\lhu\gamma$ is generated by the elements
\begin{equation}
\label{eq3.4}
[v,w]+\lambda[v,w]-v_{-1}w_{-1}\lambda[v_{0},w_{0}],
\end{equation}
where $v\in V_{\alpha}$, $w\in V_{\beta}$ with $\alpha>\beta$.
\end{lemma}
\begin{proof}
This is directly verified.
We only remark that for $z\in [V_{\alpha},V_{\beta}]$, $\gamma^{-1}(z)=-\lambda(z)$.
\end{proof}

\begin{notation}
For each $\alpha\in\Lambda$, let $R_{\alpha}$ be a pre-Nichols algebra of $V_{\alpha}$ with defining ideal $I_{\alpha}$.
Let
\[ J\ \ (\mathrm{resp.,}\ \ \tilde{J}) \]
denote the ideal of $TV$ generated by all $I_{\alpha}$ and $Z$ (resp., by all $I_{\alpha}$);
these are homogeneous braided bi-ideals which trivially intersect with $V$.
Define pre-Nichols algebras of $V$ by
\begin{equation}
\label{eq3.5}
R=TV/J,\quad \tilde{R}=TV/\tilde{J}.
\end{equation}
\end{notation}

For each finite subset $\{\alpha_{1},\dotsc,\alpha_{n}\}$ of $\Lambda$ with $\alpha_{1}<\dotsb<\alpha_{n}$, we have by (iterative use of) Proposition \ref{2.4} the pre-Nichols algebra
\[ R_{\alpha_{1}}\otimes\dotsb\otimes R_{\alpha_{n}} \]
defined by the ideal generated by all $I_{\alpha_{i}}$, $[V_{\alpha_{i}},V_{\alpha_{j}}]$ ($i>j$).
By taking the inductive limit along the obvious inclusions, we obtain the pre-Nichols algebra
\[ \bigotimes_{\alpha\in\Lambda}R_{\alpha} \]
defined by $J$.
This proves the following.

\begin{lemma}
\label{3.5}
$\bigotimes_{\alpha\in\Lambda}R_{\alpha}=R$.
\end{lemma}

\begin{definition}
\label{3.6}
(1) Let
\[ \cH=R\dotrtimes H \]
denote the Hopf algebra arising by bosonization from the pre-Nichols algebra $R$ of $V$ defined in (\ref{eq3.5}).
This coincides with the quotient Hopf algebra of $TV\dotrtimes H$ by the Hopf ideal generated by all $I_{\alpha}$ and the elements $[v,w]$, where $v\in V_{\alpha}$, $w\in V_{\beta}$ with $\alpha>\beta$.

(2) Let
\[ \cA\ \ (\mathrm{resp.,}\ \ \cH^{\lambda}) \]
denote the quotient algebra of $TV\dotrtimes H$ by the ideal generated by all $I_{\alpha}$ and the elements given in (\ref{eq3.3}) (resp., by all $I_{\alpha}$ and the elements given in (\ref{eq3.4})).
\end{definition}

\begin{remark}
\label{3.7}
The construction does not depend on choice of the total order on $\Lambda$, since even if $\alpha<\beta$, the elements in (\ref{eq3.3}) (resp., in (\ref{eq3.4})) remain to generate the same ideal (even to span the same subspace); see Proposition \ref{2.4} (ii).
\end{remark}

Since by Proposition \ref{2.4} (see (iii)), the image of $Z\hookrightarrow TV\twoheadrightarrow\tilde{R}$ is included in $P(\tilde{R})$, the composite induces a braided Hopf algebra map $TZ\rightarrow\tilde{R}$, whose bosonization we denote by
\[ \iota : TZ\dotrtimes H\rightarrow\tilde{R}\dotrtimes H. \]
Apply Proposition \ref{1.1} to this $\iota$, the Hopf ideal $\ga$ in (\ref{eq3.1}) and the algebra map $\gamma$ in (\ref{eq3.2}).
Since one sees from Lemma \ref{3.3} that the $L$ and the $A$ in the proposition turn into $\cH^{\lambda}$ and $\cA$, respectively, the next lemma follows.

\begin{lemma}
\label{3.8}
$\cH^{\lambda}$ is a Hopf algebra, which coincides with $\cH$ if $\lambda$ is the zero map.
The coproduct of $\tilde{R}\dotrtimes H$ induces algebra maps
\[ \cH^{\lambda}\otimes\cA\leftarrow\cA\rightarrow\cA\otimes\cH \]
by which $\cA$ is an $(\cH^{\lambda},\cH)$-comodule algebra.
\end{lemma}

For each $\alpha\in\Lambda$, the inclusion $V_{\alpha}\hookrightarrow V$ extends uniquely to a graded braided Hopf algebra map $R_{\alpha}\rightarrow\tilde{R}$.
Let
\[ e_{\alpha} : R_{\alpha}\rightarrow\cH^{\lambda},\quad f_{\alpha} : R_{\alpha}\rightarrow\cA \]
denote the composites with the natural maps $\tilde{R}\hookrightarrow\tilde{R}\dotrtimes H\twoheadrightarrow\cH^{\lambda},\cA$.
Let
\[ e_{0} : H\rightarrow\cH^{\lambda},\quad f_{0} : H\rightarrow\cA \]
denote the natural maps which factor through $\tilde{R}\dotrtimes H$.
Notice from Lemma \ref{3.5} that we have a natural identification
\begin{equation}
\label{eq3.6}
\bigotimes_{\alpha\in\Lambda}R_{\alpha}\otimes H=\cH.
\end{equation}
Construct the tensor products $\bigotimes_{\alpha}e_{\alpha}\otimes e_{0}$, $\bigotimes_{\alpha}f_{\alpha}\otimes f_{0}$, and compose them with the products of $\cH^{\lambda}$, and of $\cA$, respectively.
Let
\[ \eta : \cH=\bigotimes_{\alpha}R_{\alpha}\otimes H\rightarrow\cH^{\lambda},\quad \phi : \cH=\bigotimes_{\alpha}R_{\alpha}\otimes H\rightarrow\cA \]
denote the resulting maps, where we have used (\ref{eq3.6}).

\begin{proposition}
\label{3.9}
$\eta,\phi$ are isomorphisms preserving unit.
Moreover, $\eta$ is a coalgebra isomorphism, and the diagram
\[ \begin{CD}
\cH\otimes\cH @<{\Delta}<< \cH @>{\Delta}>> \cH\otimes\cH \\
@V{\eta\otimes\phi}VV @VV{\phi}V @VV{\phi\otimes\id}V \\
\cH^{\lambda}\otimes\cA @<<< \cA @>>> \cA\otimes\cH
\end{CD} \]
commutes, where the arrows on the bottom denote the bicomodule structure on $\cA$.
\end{proposition}
\begin{proof}
It is easy to see that $\eta,\phi$ preserve the unit, $\eta$ is a coalgebra map and the diagram commutes;
for the last, insert the coproduct of $\tilde{R}\dotrtimes H$ into the middle level.
Suppose that we have proved $\phi$ is bijective.
Since then $\cA\neq 0$ in particular, it follows by Proposition \ref{1.1} (3) that $\cA$ is $(\cH^{\lambda},\cH)$-biGalois.
By \cite[Theorem 9]{DT1}, $\phi$ is necessarily invertible, so that $\cA$ is right $\cH$-cleft with section $\phi$.
The invertibility of $\phi$ alternatively follows from \cite[Lemma 5.2.10]{Mo}, since $\phi$ is obviously invertible, restricted to $H$ which includes the coradical of $\cH$.
It follows by Lemma \ref{1.5} that $\eta$ is an isomorphism;
this will complete the proof.

We aim to prove that $\phi$ is bijective.
First, we remark that the generators of the ideal in $TV\dotrtimes H$ by which we define $\cA$, that is, all $I_{\alpha}$ and the elements given in (\ref{eq3.3}), span an $H$-submodule of $TV$.
Let $P$ denote the ideal in $TV$ which those generators generate, and set $A=TV/P$.
Since $P\subset TV$ is an $H$-submodule, we see $HP=PH$ in $TV\dotrtimes H$, which implies $A\rtimes H=\cA$.
The map $\phi$ restricts to $\phi|_{R} : R=\bigotimes_{\alpha}R_{\alpha}\rightarrow A$, which we denote by $\phi'$, and the tensor product $\phi'\otimes\id_{H}$ coincides with $\phi$.
Therefore, it suffices to prove that $\phi'$ is bijective.
Next, for each $R_{\alpha}$, let $0=I_{\alpha,0}\subset I_{\alpha,1}\subset\dotsb$ denote the ascending chain of ideals in $T(V_{\alpha})$ as given in (\ref{eq2.1}), and set $R_{\alpha,n}=T(V_{\alpha})/I_{\alpha,n}$.
For each $n\geq 0$, let $P_{n}$ denote the ideal of $TV$ generated by all $I_{\alpha,n}$ ($\alpha\in\Lambda$) together with the elements in (\ref{eq3.3}), and set $A_{n}=TV/P_{n}$.
We can define just as $\phi'$, maps
\[ \phi_{n}' : \bigotimes_{\alpha}R_{\alpha,n}\rightarrow A_{n}\quad (n=0,1,\dotsc). \]
Since we see that $\phi'$ is the inductive limit of $\phi_{n}'$, it suffices to prove (by induction on $n$) that $\phi_{n}'$ are bijective.

Suppose $n=0$.
To prove that
\[ \phi_{0}' : \bigotimes_{\alpha}T(V_{\alpha})\rightarrow A_{0} \]
is bijective, choose a basis $X_{\alpha}$ of $V_{\alpha}$ for each $\alpha\in\Lambda$, and set $X=\bigsqcup_{\alpha}X_{\alpha}$, which is a basis of $V$.
We aim to prove the following, by using Bergman's Diamond Lemma \cite[Theorem 1.2]{B}: $A_{0}$ has a basis consisting of those monomials in $X$ in which any subword $xy$ of length $2$ is consistent with the order on $\Lambda$ (that is, $\alpha\leq\beta$ whenever $x\in X_{\alpha}$, $y\in X_{\beta}$).
This will obviously implies the bijectivity of $\phi_{0}'$.
For the purpose we introduce a partial order among the monomials $w$ in $X$, by setting $w<w'$, if $\mathop{\mathrm{length}}w<\mathop{\mathrm{length}}w'$, or if $w$ has the same length as $w'$, but has a smaller number of {\em misordered pairs}.
Here a {\em misordered pair} in $w$ is such a pair $(x,y)$ in $w=\cdots x\cdots y\cdots$ that is inconsistent with the order on $\Lambda$ (that is $\alpha>\beta$, provided $x\in X_{\alpha}$, $y\in X_{\beta}$).
The algebra $A_{0}$ is defined by the generators $X$ and the relations
\begin{equation}
\label{eq3.7}
xy=(x_{-1}\rhu y)x_{0}-\lambda[x,y],
\end{equation}
where $x\in X_{\alpha}$, $y\in X_{\beta}$ with $\alpha>\beta$.
Here we understand that $(x_{-1}\rhu y)x_{0}$ is presented as a linear combination of products of elements from $X_{\beta}\times X_{\alpha}$.
We can naturally regard (\ref{eq3.7}) as a reduction system satisfying the assumptions of \cite[Theorem 1.2]{B}.
To obtain the desired result from \cite[Theorem 1.2]{B}, it remains to prove that the overlap ambiguities which arise when we reduce
\begin{equation}
\label{eq3.8}
xyz\quad (x\in X_{\alpha},\ y\in X_{\beta},\ z\in X_{\gamma};\ \alpha>\beta>\gamma)
\end{equation}
are resolvable.
We present the two terms
\[ (x_{-1}\rightharpoonup y)x_{0},\quad \lambda[x,y] \]
on the right-hand side of (\ref{eq3.7}), respectively by the diagrams
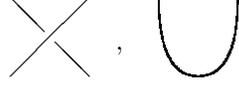
\begin{center}
\raisebox{-3.5ex}{\begin{picture}(50,40)
\put(10,5){\line(1,1){30}}
\put(10,35){\line(1,-1){13}}
\put(27,18){\line(1,-1){13}}
\end{picture}},
\raisebox{-3.5ex}{\begin{picture}(50,40)
\qbezier(10,35)(10,5)(25,5)
\qbezier(25,5)(40,5)(40,35)
\end{picture}}
\end{center}
without specifying $x,y$.
Then we reduce (\ref{eq3.8}) in two ways so that
\begin{eqnarray*}
(xy)z & = &
\raisebox{-3.5ex}{\begin{picture}(40,40)
\qbezier(5,35)(7,33)(9,31)
\qbezier(11,29)(15,25)(19,21)
\qbezier(21,19)(28,12)(35,5)
\put(5,5){\line(1,1){30}}
\qbezier(15,35)(12,32)(9,29)
\qbezier(9,29)(3,23)(12,14)
\qbezier(14,12)(18,8)(21,5)
\end{picture}}
{}-\raisebox{-3.5ex}{\begin{picture}(40,40)
\qbezier(5,35)(5,5)(15,5)
\qbezier(15,5)(25,5)(25,35)
\put(35,5){\line(0,1){30}}
\end{picture}}
{}-\raisebox{-3.5ex}{\begin{picture}(55,40)
\put(35,35){\line(-1,-1){30}}
\qbezier(10,35)(10,23)(14,17)
\qbezier(16,14)(23,5)(30,5)
\qbezier(30,5)(50,5)(50,35)
\end{picture}}
{}-\raisebox{-3.5ex}{\begin{picture}(55,40)
\qbezier(35,5)(33,7)(29,11)
\qbezier(26,14)(21,19)(16,24)
\qbezier(14,26)(10,30)(5,35)
\qbezier(50,35)(45,30)(26,11)
\qbezier(26,11)(15,0)(10,5)
\qbezier(25,35)(21,31)(13,23)
\qbezier(10,5)(0,10)(13,23)
\end{picture}}, \\
x(yz) & = &
\raisebox{-3.5ex}{\begin{picture}(40,40)
\qbezier(35,5)(33,7)(31,9)
\qbezier(29,11)(25,15)(21,19)
\qbezier(19,21)(12,28)(5,35)
\put(35,35){\line(-1,-1){30}}
\qbezier(25,5)(28,8)(31,11)
\qbezier(31,11)(37,17)(28,26)
\qbezier(26,28)(22,32)(19,35)
\end{picture}}
{}-\raisebox{-3.5ex}{\begin{picture}(40,40)
\qbezier(35,35)(35,5)(25,5)
\qbezier(25,5)(15,5)(15,35)
\put(5,5){\line(0,1){30}}
\end{picture}}
{}-\raisebox{-3.5ex}{\begin{picture}(55,40)
\qbezier(5,35)(5,5)(23,5)
\qbezier(23,5)(41,5)(41,35)
\put(20,35){\line(1,-1){17}}
\put(40,15){\line(1,-1){10}}
\end{picture}}
{}-\raisebox{-3.5ex}{\begin{picture}(60,40)
\put(50,35){\line(-1,-1){30}}
\qbezier(10,35)(20,23)(28,15)
\qbezier(30,13)(38,5)(45,5)
\qbezier(45,5)(58,5)(41,24)
\qbezier(39,26)(34,31)(30,35)
\end{picture}}.
\end{eqnarray*}
The first terms coincide by the braid relation.
Since $\lambda$ is $H$-linear, we see
\[ \raisebox{-3.5ex}{\begin{picture}(55,40)
\qbezier(35,5)(33,7)(29,11)
\qbezier(26,14)(21,19)(16,24)
\qbezier(14,26)(10,30)(5,35)
\qbezier(50,35)(45,30)(26,11)
\qbezier(26,11)(15,0)(10,5)
\qbezier(25,35)(21,31)(13,23)
\qbezier(10,5)(0,10)(13,23)
\end{picture}}
=\ \ \raisebox{-3.5ex}{\begin{picture}(40,40)
\qbezier(35,35)(35,5)(25,5)
\qbezier(25,5)(15,5)(15,35)
\put(5,5){\line(0,1){30}}
\end{picture}}. \]
Combining this with the symmetry assumption, we see
\[ \raisebox{-3.5ex}{\begin{picture}(55,40)
\qbezier(5,35)(5,5)(23,5)
\qbezier(23,5)(41,5)(41,35)
\put(20,35){\line(1,-1){17}}
\put(40,15){\line(1,-1){10}}
\end{picture}}
=\raisebox{-3.5ex}{\begin{picture}(45,40)
\qbezier(40,35)(35,5)(20,5)
\qbezier(20,5)(5,5)(15,15)
\qbezier(15,15)(20,20)(15,25)
\qbezier(13,27)(9,31)(5,35)
\qbezier(23,35)(19,31)(14,26)
\qbezier(14,26)(8,20)(14,17)
\qbezier(18,15)(22,13)(28,10)
\qbezier(32,8)(34,7)(38,5)
\end{picture}}
=\raisebox{-3.5ex}{\begin{picture}(55,40)
\put(35,35){\line(-1,-1){30}}
\qbezier(10,35)(10,23)(14,17)
\qbezier(16,14)(23,5)(30,5)
\qbezier(30,5)(50,5)(50,35)
\end{picture}}, \]
\[ \raisebox{-3.5ex}{\begin{picture}(55,40)
\put(45,35){\line(-1,-1){30}}
\qbezier(5,35)(15,23)(23,15)
\qbezier(25,13)(33,5)(40,5)
\qbezier(40,5)(53,5)(36,24)
\qbezier(34,26)(29,31)(25,35)
\end{picture}}
=\raisebox{-3.5ex}{\begin{picture}(55,40)
\qbezier(35,35)(5,20)(13,14)
\qbezier(17,12)(19,11)(21,10)
\qbezier(25,8)(27,7)(31,5)
\qbezier(5,35)(10,30)(15,25)
\qbezier(18,22)(21,19)(15,13)
\qbezier(15,13)(7,5)(18,5)
\qbezier(18,5)(20,5)(25,10)
\qbezier(25,10)(35,20)(27,28)
\qbezier(24,31)(22,33)(20,35)
\end{picture}}
=\raisebox{-3.5ex}{\begin{picture}(40,40)
\qbezier(5,35)(5,5)(15,5)
\qbezier(15,5)(25,5)(25,35)
\put(35,5){\line(0,1){30}}
\end{picture}}. \]
It results that the ambiguities are resolved, as desired.

Suppose that $\phi_{n}'$ has been proved to be bijective.
For simplicity let us write $\bar{R}_{\alpha}$, $\bar{A}$ for $R_{\alpha,n}$, $A_{n}$.
We identify $\bigotimes_{\alpha}\bar{R}_{\alpha}=\bar{A}$ through $\phi_{n}'$, and thereby regard $\bar{R}_{\alpha}\subset\bar{A}$.
Set $\bar{I}_{\alpha}=I_{\alpha,n+1}/I_{\alpha,n}$.
To prove that $\phi_{n+1}'$ is bijective, it suffices to prove that for each $\alpha\in\Lambda$, the ideal $(\bar{I}_{\alpha})$ of $\bar{A}$ generated by $\bar{I}_{\alpha}$ equals
\begin{equation}
\label{eq3.9}
\bigotimes_{\beta<\alpha}\bar{R}_{\beta}\otimes\bar{I}_{\alpha}\otimes\bigotimes_{\beta>\alpha}\bar{R}_{\beta}\ \ (\subset\bar{A}).
\end{equation}
We see by induction on $m$ that for all $m>0$, and for all $\alpha\neq\beta$ in $\Lambda$,
\begin{equation}
\label{eq3.10}
\bar{R}_{\alpha}(m)\bar{R}_{\beta}(1)+\bar{R}_{\alpha}(m-1)=\bar{R}_{\beta}(1)\bar{R}_{\alpha}(m)+\bar{R}_{\alpha}(m-1)\ \ \mbox{in}\ \ \bar{A}.
\end{equation}
Since it follows by induction on $m$ that $\bar{R}_{\alpha}\bar{R}_{\beta}(m)\subset\bar{R}_{\beta}\bar{R}_{\alpha}$, $\bar{R}_{\beta}(m)\bar{R}_{\alpha}\subset\bar{R}_{\alpha}\bar{R}_{\beta}$, we have
\begin{equation}
\label{eq3.11}
\bar{R}_{\alpha}\bar{R}_{\beta}=\bar{R}_{\alpha}\bar{R}_{\beta}\ \ \mbox{in}\ \ \bar{A}
\end{equation}
for all $\alpha,\beta\in\Lambda$.
Recall from Section 2 that $\bar{I}_{\alpha}$ is generated by such a subobject, say $U_{\alpha}$, of $\bar{R}_{\alpha}$ in ${}^{H}_{H}\YD$ that is spanned by homogeneous primitives of degree $>1$;
thus, $\bar{I}_{\alpha}=\bar{R}_{\alpha}U_{\alpha}\bar{R}_{\alpha}$.
The desired $(\bar{I}_{\alpha})=(\mbox{\ref{eq3.9}})$ will follow from (\ref{eq3.11}), if we prove that for all $x\in\bar{R}_{\beta}$ ($\beta\neq\alpha$), and for all $u\in U_{\alpha}$,
\begin{equation}
\label{eq3.12}
ux=(u_{-1}\rhu x)u_{0}\ \ \mbox{in}\ \ \bar{A},
\end{equation}
which will imply that $U_{\alpha}\bar{R}_{\beta}=\bar{R}_{\beta}U_{\alpha}$.
We may suppose $x\in X_{\beta}$.
Set $\cA_{n}=A_{n}\rtimes H$.
The bijectivity of $\phi_{n}'$ implies that $\cA_{n}$ is a right $\cH_{n}$-Galois (even cleft) object; see the first paragraph of the proof.
Let $\rho : \cA_{n}\rightarrow\cA_{n}\otimes\cH_{n}$ denote the comodule structure.
Let $\tilde{R}_{n}$ denote the quotient braided Hopf algebra of $TV$ by the ideal generated by all $I_{\alpha,n}$ ($\alpha\in\Lambda$).
Notice that $\cH_{n}$ and $\cA_{n}$ are quotients of $\tilde{R}_{n}\dotrtimes H$, and $\rho$ is induced from the coproduct of $\tilde{R}_{n}\dotrtimes H$.
For $u\in U_{\alpha}$ and $x\in X_{\beta}$ with $\alpha\neq\beta$, denote their images of the natural maps $\bar{R}_{\alpha},\bar{R}_{\beta}\rightarrow\tilde{R}_{n}$ by the same symbols.
We apply Remark 2.5 to see that the element $[u,x]$ in $\tilde{R}_{n}$ (with the abused notation (\ref{eq2.3})) is a primitive, which is zero in $\cH_{n}$, by construction of $\cH_{n}$.
It follows that the image of $[u,x]$ in $\cA_{n}$, which we denote by $v$, is coinvariant under $\rho$, that is, $\rho(v)=v\otimes 1$.
This implies $v\in\cA_{n}^{\co\cH_{n}}=k$.
But, since $u\in\bigoplus_{m>1}\bar{R}_{\alpha}(m)$, $x\in\bar{R}_{\beta}(1)$, we conclude from (\ref{eq3.10}) that $v=0$.
This is precisely the desired (\ref{eq3.12}).
\end{proof}

\begin{theorem}
\label{3.10}
$\cA$ is an $(\cH^{\lambda},\cH)$-bicleft object which has $\phi : \cH\xrightarrow{\simeq}\cA$ as a right $\cH$-colinear section, and $\phi\circ\eta^{-1} : \cH^{\lambda}\xrightarrow{\simeq}\cA$ as a left $\cH^{\lambda}$-colinear section.
Let $\sigma : \cH\otimes\cH\rightarrow k$ denote the $2$-cocycle arising from $(\cH,\phi)$; see Proposition 1.2.
Then, $\eta$ turns into a Hopf algebra isomorphism $\cH^{\sigma}\xrightarrow{\simeq}\cH^{\lambda}$, so that $\cH^{\lambda}$ is a cocycle deformation of $\cH$.
\end{theorem}
\begin{proof}
This follows from Lemma \ref{1.5}, the last proposition and the first paragraph of the last proof.
\end{proof}

Notice that $\tilde{R}\dotrtimes H$ is a graded Hopf algebra with the $n$th component $\tilde{R}(n)\otimes H$, whence it is a filtered Hopf algebra with the $n$th term $\bigoplus_{i=0}^{n}\tilde{R}(i)\otimes H$.
$\cH=R\dotrtimes H$ is a quotient graded Hopf algebra of $\tilde{R}\dotrtimes H$.
Let $\pi : \tilde{R}\dotrtimes H\rightarrow\cH^{\lambda}$ denote the quotient map.
Then, $\cH^{\lambda}$ is a filtered Hopf algebra with respect to the filtration inherited through $\pi$.
Let $\gr\cH^{\lambda}$ denote the associated graded Hopf algebra.
The graded Hopf algebra map $\gr\pi : \tilde{R}\dotrtimes H\rightarrow\gr\cH^{\lambda}$ associated to the now filtered $\pi$ is surjective.

\begin{corollary}
\label{3.11}
$\gr\pi$ induces an isomorphism $\cH\xrightarrow{\simeq}\gr\cH^{\lambda}$ of graded Hopf algebras.
\end{corollary}
\begin{proof}
$\cA$ is naturally filtered, just as $\cH^{\lambda}$ above, so that the bicomodule structure maps $\cH^{\lambda}\otimes\cA\leftarrow\cA\rightarrow\cA\otimes\cH$ are filtered algebra maps.
The associated graded algebra maps makes $\gr\cA$ into a $(\gr\cH^{\lambda},\cH)$-bicleft object, since the graded isomorphisms $\gr\phi$, $\gr(\phi\circ\eta^{-1})$ associated to the filtered $\phi$, $\phi\circ\eta^{-1}$ give sections;
see the proof of \cite[Proposition 4.4]{M3}.
Since one sees that $[v,w]=0$ in $\gr\cH^{\lambda}$, where $v\in V_{\alpha}$, $w\in V_{\beta}$ with $\alpha>\beta$, it follows that $\gr\pi$ factors through a graded Hopf algebra map
\begin{equation}
\label{eq3.13}
\cH\rightarrow\gr\cH^{\lambda}.
\end{equation}
Similarly, the natural graded map $\tilde{R}\dotrtimes H\rightarrow\gr\cA$ of right $\cH$-comodule algebras factors through
\begin{equation}
\label{eq3.14}
\cH\rightarrow\gr\cA,
\end{equation}
which, being a map of $\cH$-cleft objects, must be an isomorphism; see \cite[Lemma 1.3]{M1}.
Recall that $\gr\cA$ is left $\gr\cH^{\lambda}$-Galois, and $\cH$ is left $\cH$-Galois.
The isomorphism (\ref{eq3.14}) is compatible with the left coactions by $\cH$, $\gr\cH^{\lambda}$ through (\ref{eq3.13}).
Just as in the same way of proving in Lemma 1.5 that $\eta$ is an isomorphism, we see that (\ref{eq3.13}) is an isomorphism.
\end{proof}

\begin{remark}
\label{3.12}
(1) Since the coalgebra isomorphism $\eta : \cH\xrightarrow{\simeq}\cH^{\lambda}$ preserves the filtration, $\cH^{\lambda}$ is, as a coalgebra, graded so that $\cH^{\lambda}=\gr\cH^{\lambda}$.
Regarded as a map $\cH\rightarrow\gr\cH^{\lambda}$, $\eta$ coincides with the graded Hopf algebra isomorphism (\ref{eq3.13}).

(2) Suppose that the Hopf algebra $H$ is cosemisimple, and $R_{\alpha}$ are all Nichols algebras.
Then the graded Hopf algebra $\cH$ is coradically graded \cite[p.\ 15]{AS1}.
It follows by (1) above that the natural filtration on $\cH^{\lambda}$ then coincides with the coradical filtration \cite[p.\ 60]{Mo}.
As will be seen from Example \ref{5.4} below, this can apply to the quantized enveloping algebra $U_{q}$, when $q$ is not a root of $1$.
The result generalizes \cite[Theorem B]{CM}, which determines the coradical filtration on $U_{q}$ in some restricted situation.
\end{remark}

\section{Quantum double construction on pre-Nichols algebras}

Let $H$ be a Hopf algebra with bijective antipode, and let $\sigma : H\otimes H\rightarrow k$ be a $2$-cocycle.
Suppose $V\in{}^{H}_{H}\YD$.
Since $H=H^{\sigma}$ as a coalgebra, $V$ may be regarded as a left $H^{\sigma}$-comodule.
Denote this by ${}_{\sigma}V$, and endow ${}_{\sigma}V$ with the left $H^{\sigma}$-action $\rhu_{\sigma}$ defined by
\[ h\rhu_{\sigma}v=\sigma(h_{1},v_{-2})\sigma^{-1}(h_{2}v_{-1}\cS(h_{4}),h_{5})h_{3}\rhu v_{0}, \]
where $h\in H^{\sigma}$, $v\in{}_{\sigma}V$.
The following is a well-known result; cf.\ \cite[Proposition 1.1]{M3}.

\begin{proposition}
\label{4.1}
We have ${}_{\sigma}V\in{}^{H^{\sigma}}_{H^{\sigma}}\YD$.
Moreover, $V\mapsto{}_{\sigma}V$ gives a tensor equivalence ${}^{H}_{H}\YD\approx{}^{H^{\sigma}}_{H^{\sigma}}\YD$ preserving the braiding, to which is associated the tensor structure
\[ {}_{\sigma}V\otimes{}_{\sigma}W\xrightarrow{\simeq}{}_{\sigma}(V\otimes W),\quad v\otimes w\mapsto\sigma(v_{-1},w_{-1})v_{0}\otimes w_{0}, \]
where $V,W\in{}^{H}_{H}\YD$.
\end{proposition}

Suppose that $R$ is a braided Hopf algebra in ${}^{H}_{H}\YD$.
By the tensor equivalence above, ${}_{\sigma}R$ turns into a braided Hopf algebra in ${}^{H^{\sigma}}_{H^{\sigma}}\YD$.

\begin{proposition}
\label{4.2}
(1) Explicitly, ${}_{\sigma}R$ has the deformed product and coproduct defined by
\begin{eqnarray*}
a\cdot b & := & \sigma(a_{-1},b_{-1})a_{0}b_{0}, \\
\Delta(a) & := & \sigma^{-1}((a_{1})_{-1},(a_{2})_{-1})(a_{1})_{0}\otimes(a_{2})_{0},
\end{eqnarray*}
where $a,b\in R$, while it has the same unit, counit and antipode as $R$.

(2) Regard $\sigma$ as a $2$-cocycle for $R\dotrtimes H$ along the obvious projection $R\dotrtimes H\rightarrow H$, and construct the cocycle deformation $(R\dotrtimes H)^{\sigma}$.
Then we have a Hopf algebra isomorphism
\[ {}_{\sigma}R\dotrtimes H^{\sigma}\xrightarrow{\simeq}(R\dotrtimes H)^{\sigma}\]
given by $a\otimes h\mapsto\sigma(a_{-1},h_{1})a_{0}\otimes h_{2}$.
\end{proposition}

This is directly verified; cf.\ \cite[Propositions 1.12, 1.13]{M3}.

Let $J,K$ be Hopf algebras with bijective antipode.
Let $V\in{}^{J}_{J}\YD$, $W\in{}^{K}_{K}\YD$.
Let $R$ be a pre-Nichols algebra of $V$ in ${}^{J}_{J}\YD$, and let $S$ be a pre-Nichols algebra of $W$ in ${}^{K}_{K}\YD$.
The following is a reformulation of \cite[Theorem 8.3]{RS1} due to Radford and Schneider, in the slightly generalized context of pre-Nichols algebras.

\begin{theorem}
\label{4.3}
Suppose that a skew pairing $\tau : J\otimes K\rightarrow k$ is given.
Let $\lambda : W\otimes V\rightarrow k$ be a linear map such that
\begin{eqnarray}
\label{eq4.1}
\lambda(x\rhu w,v) & = & \tau(v_{-1},\cS(x))\lambda(w,v_{0}), \\
\label{eq4.2}
\lambda(w,a\rhu v) & = & \tau(a,w_{-1})\lambda(w_{0},v)
\end{eqnarray}
for all $a\in J$, $x\in K$, $v\in V$, $w\in W$.
Then, $\tau$ extends uniquely to a skew pairing
\[ \tau : (R\dotrtimes J)\otimes(S\dotrtimes K)\rightarrow k \]
such that
\[ \begin{array}{l}
\tau(a,w)=0=\tau(v,x), \vspace{5pt} \\
\tau(v,w)=-\lambda(w,v),
\end{array} \]
where $a\in J$, $x\in K$, $v\in V$, $w\in W$.
By the corresponding $2$-cocycle (see Proposition \ref{1.6}), the product on $(R\dotrtimes J)\otimes(S\dotrtimes K)$ is deformed so that
\begin{eqnarray}
\label{eq4.3}
x\cdot a & = & \tau(a_{1},x_{1})a_{2}\otimes x_{2}\tau(\cS(a_{3}),x_{3}), \\
\label{eq4.4}
x\cdot v & = & \tau(v_{-1},x_{1})v_{0}\otimes x_{2}, \\
\label{eq4.5}
w\cdot a & = & \tau(a_{1},w_{-1})a_{2}\otimes w_{0}, \\
\label{eq4.6}
w\cdot v & = & \tau(v_{-1},w_{-1})v_{0}\otimes w_{0}-\lambda(w,v)+w_{-1}\cdot v_{-1}\lambda(w_{0},v_{0}),
\end{eqnarray}
where $a\in J$, $x\in K$, $v\in V$, $w\in W$.
\end{theorem}
\begin{proof}
Set $H=J\otimes K$, the tensor-product Hopf algebra.
Let $\sigma : H\otimes H\rightarrow k$ denote the $2$-cocycle corresponding to $\tau$.
Regarding $V,W\in{}^{H}_{H}\YD$ in the obvious way, we have ${}_{\sigma}V$, ${}_{\sigma}W$ in ${}^{H^{\sigma}}_{H^{\sigma}}\YD$, which are symmetric to each other since $V$ and $W$ are.
Notice that the actions of $J$ ($\subset H^{\sigma}$) on ${}_{\sigma}V$, and of $K$ ($\subset H^{\sigma}$) on ${}_{\sigma}W$ remain the same as the original ones, while
\begin{eqnarray*}
x\rhu_{\sigma}v & = & \tau(v_{-1},x)v_{0}\quad (x\in K,\ v\in{}_{\sigma}V), \\
a\rhu_{\sigma}w & = & \tau(a,\cS(w_{-1}))w_{0}\quad (a\in J,\ w\in{}_{\sigma}W).
\end{eqnarray*}
Moreover, ${}_{\sigma}R,{}_{\sigma}S$ remain to be $R,S$ except that the actions are deformed so as above.
Notice that the conditions required to $\lambda$ is precisely that $\lambda : {}_{\sigma}W\otimes{}_{\sigma}V\rightarrow k$ should be $H^{\sigma}$-linear.
Define
\[ \cH := ({}_{\sigma}R\otimes{}_{\sigma}S)\dotrtimes H^{\sigma}. \]
This is such a Hopf algebra that factorizes into two Hopf subalgebras, $R\dotrtimes J$, $S\dotrtimes K$; see the paragraph following Lemma \ref{1.5}.
By Theorem \ref{3.10}, the $H^{\sigma}$-linear map $\lambda$, regarded as a map $[{}_{\sigma}W,{}_{\sigma}V]\rightarrow k$ (see Remark \ref{3.2}), gives a cocycle deformation $\cH^{\lambda}$ of $\cH$.
One sees that $\cH^{\lambda}$ also factorizes so as $\cH^{\lambda}=(R\dotrtimes J)\otimes(S\dotrtimes K)$ into the two Hopf subalgebras, and obeys the product rules (\ref{eq4.3})--(\ref{eq4.6}); for (\ref{eq4.5}), recall $\tau^{-1}(a,x)=\tau(a,\cS^{-1}(x))$.
The $(\cH^{\lambda},\cH)$-bicleft object, say $\cA$, given by Theorem \ref{3.10} (see also Definition \ref{3.6}), is defined on the factorized algebra $(R\dotrtimes J)\otimes (S\dotrtimes K)$, and by the product rules (\ref{eq4.3}), (\ref{eq4.4}), (\ref{eq4.5}) and
\begin{equation}
\label{eq4.7}
w\cdot v=\tau(v_{-1},w_{-1})v_{0}\otimes w_{0}-\lambda(w,v)\quad (v\in V,\ w\in W).
\end{equation}

Set $\cG=(R\dotrtimes J)\otimes(S\dotrtimes K)$, the tensor product Hopf algebra.
Proposition \ref{4.2} (2) gives a natural identification $\cH=\cG^{\sigma}$.
Therefore, $\sigma^{-1}$ is a $2$-cocycle for $\cH$, and $\cH^{\sigma^{-1}}=\cG$.
We have a tensor equivalence between right comodule categories,
\[ U\mapsto U_{\sigma},\quad \cM^{\cH}\approx\cM^{\cG}, \]
just as $V\mapsto{}_{\sigma}V$ in Proposition \ref{4.1}, but on the opposite side.
Notice that $\cA_{\sigma}$ is an $(\cH^{\lambda},\cG)$-bicleft object, which is defined on $(R\dotrtimes J)\otimes(S\dotrtimes K)$, and by the product rules
\[ x\cdot a=\tau(a_{1},x_{1})a_{2}\otimes x_{2}\quad (a\in J,\ x\in K), \]
(\ref{eq4.4}), (\ref{eq4.5}) and (\ref{eq4.7}).
By applying the counit to the right-hand sides of these four equations, we see from Proposition \ref{1.6} that $\cA_{\sigma}$ together with the obvious section $\cG\xrightarrow{\simeq}\cA_{\sigma}$ gives rise to such a $2$-cocycle that corresponds to the skew pairing $\tau$ claimed above to exist, and that $\cH^{\lambda}$ is the cocycle deformation of $\cG$ by that $2$-cocycle.
\end{proof}

\section{How our results fit in with quantized enveloping algebras}

Let us see how our results are specialized to the situation in which the quantized enveloping algebras are involved.
Let $\Gamma$ be an abelian group, and let $\hat{\Gamma}$ denote the dual group of all group maps $\Gamma\rightarrow k^{\times}=k\setminus 0$.
Let $k\Gamma$ denote the group Hopf algebra.
A left Yetter-Drinfeld module, or an object in ${}^{k\Gamma}_{k\Gamma}\YD$, is precisely a left $k\Gamma$-module $V$ which is at the same time $\Gamma$-graded, $V=\bigoplus_{g\in\Gamma}V_{g}$, so that for each $g\in\Gamma$, the $g$-component $V_{g}$ in $V$ is $k\Gamma$-stable.
A pair $(g,\chi)$ in $\Gamma\times\hat{\Gamma}$ defines a one-dimensional object $V=kx$ in ${}^{k\Gamma}_{k\Gamma}\YD$, by
\[ h\rhu x=\chi(h)x\quad (h\in\Gamma),\quad V=V_{g}. \]
Every one-dimensional object in ${}^{k\Gamma}_{k\Gamma}\YD$ arises uniquely from a pair in $\Gamma\times\hat{\Gamma}$. \\

\noindent
{\rm\bf Definition 5.1 [9].}
An object in $V$ ($\neq 0$) in ${}^{k\Gamma}_{k\Gamma}\YD$ is said to be {\em of diagonal type}, if it is a (direct) sum of one-dimensional subobjects, or in other words if it has a basis $(x_{i})$ such that each $x_{i}$ spans a subobject.
\addtocounter{theorem}{1} \\

Let $V=\bigoplus_{i\in X}kx_{i}$ be such as above, where the basis $(x_{i})$ is supposed to be indexed by a set $X$.
Suppose that $kx_{i}$ corresponds to the pair $(g_{i},\chi_{i})$ in $\Gamma\times\hat{\Gamma}$.
By convention (see [9,10]), we write
\[ q_{ij}=\chi_{j}(g_{i}). \]
By using the associated braiding $c : V\otimes V\xrightarrow{\simeq}V\otimes V$, the scalar $q_{ij}$ is characterized as the coefficient in
\[ c(x_{i}\otimes x_{j})=q_{ij}x_{j}\otimes x_{i}. \]
The braided commutator is given by
\[ [x_{i},x_{j}]=x_{i}\otimes x_{j}-q_{ij}x_{j}\otimes x_{i}. \]

We suppose that $X$ is a disjoint union $X=\bigsqcup_{\alpha\in\Lambda}X_{\alpha}$ of non-empty subsets $X_{\alpha}$ indexed by a set $\Lambda$.
Introduce a total order onto $\Lambda$.
Set
\[ V_{\alpha}=\bigoplus_{i\in X_{\alpha}}kx_{i}. \]
This is an object in ${}^{k\Gamma}_{k\Gamma}\YD$ of diagonal type, such that $V=\bigoplus_{\alpha\in\Lambda}V_{\alpha}$.
Let $|i|$ denote $\alpha$ when $i\in X_{\alpha}$.
Suppose that if $\alpha\neq\beta$ in $\Lambda$, then $V_{\alpha}$ and $V_{\beta}$ are symmetric.
This means that
\[ q_{ij}q_{ji}=1\ \ \mbox{if}\ \ |i|\neq|j|. \]
For each pair $\alpha>\beta$ in $\Lambda$, choose a $k\Gamma$-linear map $\lambda : [V_{\alpha},V_{\beta}]\rightarrow k$.
Set
\[ \lambda_{ij}=\lambda[x_{i},x_{j}]\quad (|i|>|j|). \]
Since $\Gamma$ acts on $[x_{i},x_{j}]$ via $\chi_{i}\chi_{j}$, it follows that
\begin{equation}
\label{eq5.1}
\lambda_{ij}=0\ \ \mbox{if}\ \ \chi_{i}\chi_{j}\neq 1.
\end{equation}
We see that choosing $\lambda$ is equivalent to choosing those parameters $\lambda_{ij}$ ($|i|>|j|$) in $k$ which satisfy (\ref{eq5.1}).
It is reasonable to define $\lambda_{ji}=-q_{ji}\lambda_{ij}$, since $[x_{j},x_{i}]=-q_{ji}[x_{i},x_{j}]$.
Andruskiewitsch and Schneider [9,10] call those parameters (more precisely, $-\lambda_{ij}$ in minus sign) {\em linking parameters}.

For each $\alpha\in\Lambda$, choose a pre-Nichols algebra $R_{\alpha}$ of $V_{\alpha}$, and let $I_{\alpha}$ denote its defining ideal.
By Lemma \ref{3.5}, the tensor product
\[ R:=\bigotimes_{\alpha\in\Lambda}R_{\alpha} \]
taken along the order on $\Lambda$ is a pre-Nichols algebra of $V$, whose bosonization we denote as before, by
\[ \cH=R\dotrtimes k\Gamma. \]
Recall that since $R$ is graded, $\cH$ is a graded Hopf algebra with $\cH(n)=R(n)\otimes k\Gamma$.
Let $\cH^{\lambda}$ denote the quotient algebra of $TV\dotrtimes k\Gamma$ by the ideal generated by all $I_{\alpha}$ together with the elements
\[ x_{i}\otimes x_{j}-q_{ij}x_{j}\otimes x_{i}-\lambda_{ij}(g_{i}g_{j}-1)\quad (|i|>|j|). \]
By Lemma \ref{3.8}, $\cH^{\lambda}$ is a quotient Hopf algebra, which coincides with $\cH$ if all $\lambda_{ij}=0$.
We regard $\cH^{\lambda}$ as a filtered Hopf algebra with respect to the natural filtration inherited from $TV\dotrtimes k\Gamma$.
Compose the tensor product of the natural algebra maps $R_{\alpha}\rightarrow\cH^{\lambda}$, $k\Gamma\rightarrow\cH^{\lambda}$ with the product on $\cH^{\lambda}$, and let
\[ \eta : \cH=\bigotimes_{\alpha\in\Lambda}R_{\alpha}\otimes k\Gamma\rightarrow\bigotimes_{\alpha\in\Lambda}\cH^{\lambda}\otimes\cH^{\lambda}\rightarrow\cH^{\lambda} \]
denote the resulting map.
Proposition \ref{3.9}, Theorem \ref{3.10} and Remark \ref{3.12} (actually, some part of them) are specialized as follows.

\begin{theorem}
\label{5.2}
(1) $\eta : \cH\rightarrow\cH^{\lambda}$ is a coalgebra isomorphism.
It is in fact an isomorphism of filtered coalgebras, whence $\cH^{\lambda}$ is, as a coalgebra, graded so that $\cH^{\lambda}=\gr\cH^{\lambda}$.

(2) $\eta : \cH\xrightarrow{\simeq}\cH^{\lambda}$ turns into a Hopf algebra isomorphism, by replacing $\cH$ with some cocycle deformation $\cH^{\sigma}$.

(3) By (1), we can regard $\eta$ as a map $\cH\rightarrow\gr\cH^{\lambda}$.
In this case, $\eta$ is an isomorphism of graded Hopf algebras.
\end{theorem}

Instead of restating Theorem \ref{4.3} in the present special situation, we will give its consequence, Theorem \ref{5.3}.
It is a direct generalization of \cite[Theorems 4.4, 4.11]{RS2} due to Radford and Schneider.
In most parts, we will follow their argument.

Keep the notation as above.
But, we suppose that
\[ \Lambda=\{+,-\}, \]
consisting of the two elements $+,-$ with $+>-$.
To each $j\in X_{-}$, associate a symbol $g_{i}'$.
Let $\Gamma'$ denote the free abelian group on the set $\{g_{j}'\;|\;j\in X_{-}\}$.
Define $\chi_{j}'\in\hat{\Gamma}'$ ($j\in X_{-}$) by
\[ \chi_{j}'(g_{k}'):=\chi_{j}(g_{k})=q_{kj}\quad (j,k\in X_{-}). \]
We can regard $V_{-}$ (resp., $R_{-}$) as an object (resp., its pre-Nichols algebra) in ${}^{k\Gamma'}_{k\Gamma'}\YD$, by
\[ g_{k}'\rhu x_{j}=q_{kj}x_{j},\quad x_{j}\in(V_{-})_{g_{j}'}\quad (j,k\in X_{-}). \]
Define Hopf algebras by
\[ \cH_{-}=R_{-}\dotrtimes k\Gamma',\quad \cH_{+}=R_{+}\dotrtimes k\Gamma. \]
We have a bimultiplicative map
\[ \tau : \Gamma'\times\Gamma\rightarrow k^{\times},\quad \tau(g_{j}',g)=\chi_{j}(g), \]
which extends uniquely to a (skew) pairing $\tau : k\Gamma'\otimes k\Gamma\rightarrow k$ of Hopf algebras.
As for this $\tau$, the linear map
\[ \lambda : V_{+}\otimes V_{-}\rightarrow k,\quad \lambda(x_{i},x_{j})=\lambda_{ij} \]
satisfies the conditions (\ref{eq4.1}), (\ref{eq4.2}) in Theorem \ref{4.3}, which now read
\[ \chi_{i}(g)\lambda_{ij}=\chi_{j}(g^{-1})\lambda_{ij},\quad \chi_{j}(g_{k})\lambda_{ij}=\chi_{k}(g_{i})\lambda_{ij}, \]
where $i\in X_{+}$, $j\in X_{-}$, $g\in\Gamma$.
To verify these equations, we may suppose $\lambda_{ij}\neq 0$, and so $\chi_{i}\chi_{j}=1$, in which case $\chi_{i}(g)=\chi_{j}(g^{-1})$, $\chi_{j}(g_{k})=\chi_{k}(g_{i})$, indeed.
By the last cited theorem, $\tau$ extends uniquely to a skew pairing $\tau : \cH_{-}\otimes\cH_{+}\rightarrow k$, so that
\[ \tau(g_{j}',x_{i})=0=\tau(x_{j},g),\quad \tau(x_{j},x_{i})=-\lambda_{ij}, \]
where $i\in X_{+}$, $j\in X_{-}$, $g\in\Gamma$.
Let $\sigma : (\cH_{-}\otimes\cH_{+})^{\otimes 2}\rightarrow k$ denote the corresponding $2$-cocycle.
The cocycle deformation $(\cH_{-}\otimes\cH_{+})^{\sigma}$ factorizes into Hopf subalgebras, $\cH_{-},\cH_{+}$, and obeys the product rules as given by (\ref{eq4.3})--(\ref{eq4.6}), which now read
\begin{eqnarray*}
g\cdot g_{j}' & = & g_{j}'\otimes g, \\
g\cdot x_{j} & = & \chi_{j}(g)x_{j}\otimes g, \\
x_{i}\cdot g_{j}' & = & \chi_{j}(g_{i})g_{j}'\otimes x_{i}, \\
x_{i}\cdot x_{j} & = & \chi_{j}(g_{i})x_{j}\otimes x_{i}+\lambda_{ij}(g_{j}'\otimes g_{i}-1),
\end{eqnarray*}
where $i\in X_{+}$, $j\in X_{-}$, $g\in\Gamma$.
We see that for each $j\in X_{-}$, $g_{j}'\otimes g_{j}^{-1}$ is a central grouplike in $(\cH_{-}\otimes\cH_{+})^{\sigma}$.
Therefore, the elements $g_{j}'\otimes g_{j}^{-1}-1$ ($j\in X_{-}$) generate a Hopf ideal, say $Q$, in $(\cH_{-}\otimes\cH_{+})^{\sigma}$.

\begin{theorem}
\label{5.3}
$x_{i}\mapsto 1\otimes x_{i}$ ($i\in X_{+}$), $x_{j}\mapsto x_{j}\otimes 1$ ($j\in X_{-}$) and $g\mapsto 1\otimes g$ ($g\in\Gamma$) give an isomorphism
\[ \cH^{\lambda}\xrightarrow{\simeq}(\cH_{-}\otimes\cH_{+})^{\sigma}/Q \]
of Hopf algebras.
\end{theorem}
\begin{proof}
By definition of $\cH^{\lambda}$, and by the product rules given above, the correspondences well define a Hopf algebra map, which is an isomorphism since we see $R_{-}\otimes R_{+}\otimes k\Gamma=(\cH_{-}\otimes\cH_{+})^{\sigma}/Q$.
\end{proof}

Radford and Schneider \cite[Theorems 4.4, 4.11]{RS2} prove this result in some restricted situation, especially such that those $(i,j)\in X_{+}\times X_{-}$ for which $\lambda_{ij}\neq 0$ appear in pair, that is, if $\lambda_{ij}\neq 0$, then $\lambda_{lj}=0=\lambda_{ik}$ whenever $l\neq i$, $k\neq j$.
This restriction can be removed.

Finally, let us explain how the quantized enveloping algebras and their analogues or generalizations are presented as our $\cH^{\lambda}$.
This will be parallel with \cite[Remarks 1.7, 1.8]{RS2}, in which Radford and Schneider explain how those algebras are presented as their own $U(\cDred,l)$.

\begin{example}
\label{5.4}
Let $A=(a_{ij})$ be an $n\times n$ generalized Cartan matrix (GCM) symmetrized by a diagonal matrix $\diag(d_{1},\dotsc,d_{n})$ with $0<d_{i}\in\bbz$.
Choose $q\in k^{\times}$.
Set $q_{i}=q^{d_{i}}$, and suppose that the order $\ord(q_{i}^{2})$ is greater than all of $1$, $-a_{ij}$ ($i\neq j$).
Let $U_{q}$ denote the associated, standard quantized enveloping algebra.
The grouplikes in $U_{q}$ form a free abelian group, say $\Gamma$, which is isomorphic to such a lattice $\gh_{\bbz}$ in the Cartan subalgebra $\gh$ realizing $A$ that satisfies appropriate assumptions;
see the beginning of \cite[Section 7]{M3}, for example.
With the standard notation, $U_{q}$ contains generators $E_{i},F_{i}$ ($1\leq i\leq n$) such that
\[ \Delta E_{i}=E_{i}\otimes 1+K_{i}\otimes E_{i},\quad \Delta F_{i}=F_{i}\otimes K_{i}^{-1}+1\otimes F_{i}, \]
where $K_{i}\in\Gamma$.
We see that the $n$ elements
\[ x_{i}:=E_{i}\quad\mbox{(resp., $x_{-i}:=F_{i}K_{i}$)}\quad (1\leq i\leq n) \]
span an object, say $V_{+}$ (resp., $V_{-}$), in ${}^{k\Gamma}_{k\Gamma}\YD$ of diagonal type, such that
\[ K_{i}\rhu x_{\pm j}=q_{i}^{\pm a_{ij}}x_{\pm j},\quad x_{\pm i}\in(V_{\pm})_{K_{i}}, \]
and it generates a pre-Nichols algebra, say $R_{+}$ (resp., $R_{-}$), defined by the quantum Serre relations;
it is known to be a Nichols algebra at least when $\ch k=0$, $q$ is not a root of $1$, and the GCM $A$ is of finite type.
Since $V_{+}$ and $V_{-}$ are symmetric, as is easily seen, we have a Hopf algebra, $\cH=(R_{-}\otimes R_{+})\dotrtimes k\Gamma$, which involves the relations
\begin{equation}
\label{eq5.2}
x_{i}x_{-j}=q_{i}^{-a_{ij}}x_{-j}x_{i}\ \ \mbox{or}\ \ E_{i}F_{j}=F_{j}E_{i}\quad (1\leq i,j\leq n).
\end{equation}
Since $\Gamma$ acts trivially on each $x_{i}\otimes x_{-i}$, we can deform $\cH$ to $\cH^{\lambda}$ by
\[ \lambda_{i,-j}=\delta_{ij}\frac{1}{q_{i}-q_{i}^{-1}}\quad (1\leq i,j\leq n). \]
In $\cH^{\lambda}$, (\ref{5.2}) is deformed so as
\begin{eqnarray*}
x_{i}x_{-j} & = & q_{i}^{-a_{ij}}x_{-j}x_{i}+\delta_{ij}\frac{K_{i}^{2}-1}{q_{i}-q_{i}^{-1}} \\
\mbox{or}\ \ \ \ E_{i}F_{j} & = & F_{j}E_{i}+\delta_{ij}\frac{K_{i}-K_{i}^{-1}}{q_{i}-q_{i}^{-1}}.
\end{eqnarray*}
Therefore we see $U_{q}=\cH^{\lambda}$.
(To be more rigorous, we must abstractly define $V_{\pm},R_{\pm},\cH^{\lambda}$ in the obvious manner as will be seen from the above argument, and prove that $x_{i}\mapsto E_{i}$, $x_{-i}\mapsto F_{i}K_{i}$ give an isomorphism $\cH^{\lambda}\xrightarrow{\simeq}U_{q}$ over $k\Gamma$.)
By Theorem \ref{5.2}, $U_{q}$ is a cocycle deformation of $\cH$;
this was proved by \cite[Theorem 7.8]{M3}, generalizing the result by Kassel and Schneider \cite{KS} in the special case when $A$ is of finite type.
By Theorem \ref{5.3}, $U_{q}$ is a quotient Hopf algebra of the quantum double
\[ ((R_{-}\dotrtimes k\Gamma')\otimes(R_{+}\dotrtimes k\Gamma))^{\sigma}, \]
where $\Gamma'$ is the free abelian group on the set $\{K_{i}\;|\;1\leq i\leq n\}$.
$R_{+}\dotrtimes k\Gamma$ equals the Hopf subalgebra $U_{q}^{\geq 0}$ of $U_{q}$ generated by $\Gamma,E_{1},\dotsc,E_{n}$.
In a modified situation (or when we suppose $\det A\neq 0$), we have $\Gamma'=\Gamma$, in which case $R_{-}\dotrtimes k\Gamma'$ equals the Hopf subalgebra $U_{q}^{\leq 0}$ of $U_{q}$ generated by $\Gamma,F_{1},\dotsc,F_{n}$.
Then the last result reads
\[ U_{q}=(U_{q}^{\leq 0}\otimes U_{q}^{\geq 0})^{\sigma}/(K_{i}\otimes 1-1\otimes K_{i}\;|\;1\leq i\leq n). \]
Joseph \cite[3.2.9]{J} chooses this as definition of $U_{q}$.
\end{example}

\begin{example}
\label{5.5}
Keep the notation as above.
Let us give two variations of $U_{q}$.
By modifying slightly the construction above, we can present those as $\cH^{\lambda}$, to which, therefore, our results can apply.

(1) The quantized enveloping superalgebras \cite{FLV}, \cite{Y}.
A main difference from the standard $U_{q}$ is that $\Gamma$ contains a grouplike of order $2$, and the $R_{\pm}$ require some defining relations other than the quantum Serre relations.

(2) Lustzig's small quantum groups $u_{q}$ \cite{L}.
This is defined as a finite-dimensional quotient of $U_{q}$, assuming that the GCM $A$ is of finite type, and $q$ is a root of $1$ of odd order $>3$.
For this $u_{q}$, the $R_{\pm}$ are Nichols algebras defined by the quantum Serre relations together with such relations of the form
\begin{equation}
\label{eq5.3}
x_{\pm\alpha}^{N}=0,
\end{equation}
where $x_{\pm\alpha}$ ($\in R_{\pm}$) are the so-called root vectors associated to each positive root $\alpha$.
\end{example}

\begin{example}
\label{5.6}
Andruskiewitsch and Schneider [9,10] defined finite-dimensional pointed Hopf algebras $u(\cD,\lambda,\mu)$ with finite abelian group $\Gamma$, say, of grouplikes, which generalize Lustzig's $u_{q}$.
They involve two sets $\lambda,\mu$ of parameters, called {\em linking parameters}, {\em root vector parameters}, respectively.
If $\lambda=\mu=0$ (that is, if all parameters are zero), $u(\cD,0,0)$ is the bosonization $R\dotrtimes k\Gamma$ of the braided tensor product $R:=R_{1}\otimes\dotsb\otimes R_{m}$, where each $R_{i}$ is the Nichols algebra of a diagonal type object, and looks likes the $R_{\pm}$ in Example \ref{5.5} (2).
(More precisely, $R_{i}$ is a cocycle deformation of $R_{\pm}$ by a group $2$-cocycle \cite[Sectin 1.2]{AS2}.)
Therefore we can take $u(\cD,0,0)$ as our $\cH$.
As a main difference from preceding examples, $\lambda$ can {\em link} components $R_{1},\dotsc,R_{m}$ possibly more than two.
On the other hand, $\mu$ plays the role of giving such relations $x_{\alpha}^{N}=u_{\alpha}(\mu)$ that generalize (\ref{eq5.3}).
If $\mu=0$, $u(\cD,\lambda,0)$ is presented as $\cH^{\lambda}$, in the same way as in the preceding examples; see \cite[Appendix]{M3}.
By Theorem \ref{5.2}, $u(\cD,\lambda,0)$ is a cocycle deformation of $u(\cD,0,0)$; this was proved by Didt \cite{Di}.
By some additional argument, it is proved by \cite[Theorem A.1]{M3} (cf.\ \cite[Theorem 3.5]{GM}) that $u(\cD,\lambda,\mu)$ in general is a cocycle deformation of $u(\cD,0,0)$.
\end{example}

\end{document}